\documentclass[pdflatex,sn-mathphys-num]{sn-jnl}

\usepackage{graphicx}%
\usepackage{multirow}%
\usepackage{amsmath,amssymb,amsfonts}%
\usepackage{amsthm}%
\usepackage{mathrsfs}%
\usepackage[title]{appendix}%
\usepackage{xcolor}%
\usepackage{textcomp}%
\usepackage{manyfoot}%
\usepackage{booktabs}%
\usepackage{algorithm}%
\usepackage{algorithmicx}%
\usepackage{algpseudocode}%
\usepackage{listings}%

\usepackage{macros-pack}

\newtheorem{theorem}{Theorem}
\newtheorem{proposition}[theorem]{Proposition}%
\newtheorem{lemma}[theorem]{Lemma}%
\newtheorem{corollary}[theorem]{Corollary}%

\theoremstyle{remark}%

\theoremstyle{definition}%
\newtheorem{definition}{Definition}%

\begin{document}

\title[Strict refinement property of connected loop-free categories]{Strict refinement property of connected loop-free categories}


\author[1]{\fnm{Aly-Bora} \sur{Ulusoy}}\email{aly-bora.ulusoy@polytechnique.edu}

\author[1]{\fnm{Emmanuel} \sur{Haucourt}}\email{haucourt@lix.polytechnique.fr}


\affil[1]{\orgdiv{Cosynus}, \orgname{\'Ecole Polytechnique}, \orgaddress{\city{Palaiseau}, \postcode{91120}, \country{France}}}




\abstract{
    In this paper we study the strict refinement property for connected partial orders 
    also known as 
    Hashimoto's Theorem.
    This property implies that any isomorphism between products of irreducible structures is determined
    is uniquely determined as a product of isomorphisms between the factors. 
    This refinement implies a sort of smallest possible decomposition for such structures.
    After a brief recall of the necessary notion we prove that Hashimoto's theorem can be 
    extended to connected loop-free categories, \ie categories with no 
    non-trivial morphisms endomorphisms.
    A special case of such categories is the category of connected 
    components, for concurrent programs without loops.
    }




\maketitle

\emph{Connected} posets (i.e. those in which for all elements \(x\) and \(y\) we 
have a sequence \(x=z_0,z_1,\ldots,z_n=y\) such that \(z_i\) and \(z_{i-1}\) are 
\emph{comparable} for every \(i\in\{1,\ldots,n\}\)) satisfy the \emph{strict refinement property} 
\cite{schroder2003ordered}:
if the following isomorphism holds in the category of posets (\(\poset\)) 
\[
\prod_{\alpha \in A}X_\alpha\quad\cong\quad\prod_{\beta \in B}Y_\beta
\] 
then we have a family of posets \(\big\{ Z_{\alpha,\beta}\:\big|\:\alpha\in A\:;\:\beta\in B \big\}\) 
such that the following isomorphisms hold for every \(\alpha\in A\) and every \(\beta\in B\).
\[
X_\alpha\quad\cong\quad\prod_{b\in B}Z_{\alpha,b}
\qquad\text{and}\qquad
Y_\beta\quad\cong\quad\prod_{a\in A}Z_{a,\beta} 
\]

A (small) category is said to be \emph{loop-free} when any two of its morphisms 
whose composite is an endomorphism are identities.
The category of loop-free categories (\(\lfcat\)), which is a full subcategory of 
the category of small categories (\(\cat\)), contains \(\poset\) as a full subcategory.
The embedding \(\poset\hookrightarrow\lfcat\) has a left adjoint obtained by identifying 
any two arrows with the same source and the same target.
The purpose of this article is to prove that \emph{connected loop-free categories} also 
satisfy the strict refinement property:
\medskip

\begin{restatethis}{theorem}{thm:big-result}
    \label{thm:big-result}
    Given a connected loop-free category $\Cat{C}$, for any two decompositions 
    $\Psi_A : \Cat{C} \to \prod_{\alpha \in A}X_\alpha$
    and $\Psi_B : \Cat{C} \to \prod_{\beta \in B}Y_\beta$, there exists a family of 
    loop-free categories \(Z_{\alpha,\beta}\) with \(\alpha\in A\) and \(\beta\in B\), 
    and for every \(\alpha\in A\) and \(\beta\in B\), decompositions 
    \[a_\alpha\::\:X_\alpha\to\prod_{\beta\in B} Z_{\alpha,\beta}
		\quad\text{ and }\quad
		b_\beta\::\:Y_\beta\to\prod_{\alpha\in A} Z_{\alpha,\beta}
		\] 
such that 
the following diagram commutes:
\begin{equation}
    \label{diag:hashi-global}
    \begin{tikzcd}[ampersand replacement=\&]
        \&  \& \mathcal{C} \arrow[lldd, "\Psi_A"'] \arrow[rrdd, "\Psi_B"] \&  \&  \\
        \&  \&                                                            \&  \&  \\
        \prod_{\alpha\in A} X_\alpha \arrow[ddd, "(a_\alpha)_{\alpha\in A}"', dashed] \arrow[rrrr, "\Psi_B\circ\Psi^{-1}_A"] \&  \&
            \&  \& \prod_{\beta \in B} Y_\beta \arrow[ddd, "(b_\beta)_{\beta\in B}", dashed] \\
        \&  \&                                                            \&  \&  \\
        \&  \&                                                            \&  \&  \\
        {\prod_{\alpha\in A} \prod_{\beta \in B} Z_{\alpha,\beta}} \arrow[rrrr, "\gamma"]  \&  \&
            \&  \& { \prod_{\beta \in B} \prod_{\alpha\in A} Z_{\alpha,\beta}}
    \end{tikzcd}
\end{equation}
    where $\gamma : \prod_{\alpha \in A}\prod_{\beta \in B} Z_{\alpha,\beta} \to \prod_{\beta \in B}\prod_{\alpha \in A} Z_{\alpha,\beta}$
    is the natural isomorphism
    sending $((z_{\alpha,\beta})_{\beta \in B})_{\alpha \in A}$ to $((z_{\alpha,\beta})_{\alpha \in A})_{\beta \in B}$.
\end{restatethis}

The structure of our proof diverges from Hashimoto's original paper \cite{hashimoto1951direct}
and instead follows more closely the presentation found in
\cite[Chapter 10]{schroder2003ordered} which has the added advantage of also constructing the elements of the
refinement.

Beyond its pure theoretical interest, this problem is related to the study of programs made 
of several sequential processes running in parallel \cite{dijkstra_1968} 
(in the sequel we just write `program').
The geometric model \(\llbracket P\rrbracket\) of a program \(P\) is 
\emph{locally ordered} \cite[Corollary 6.1]{haucourt_2018}, so each of its points has a neighborhood whose 
\emph{fundamental category} \cite[4.37, p.73]{datc} is loop-free.
The crucial facts are: 
\begin{enumerate} 
\item the functor \({\vec\pi}_1:{\mathbf C}\to\cat\), which assigns to every object of \(\mathbf C\) its fundamental category, preserves finite products, and
\item two programs \(P\) and \(Q\) do not interact with each other when we have \(\llbracket P|Q\rrbracket\cong\llbracket P\rrbracket\!\times\!\llbracket Q\rrbracket\) (in \(\mathbf C\)) \cite[Theorem 6.2]{haucourt_2018}. 
\end{enumerate}
Hence splitting a program \(P\) into subprograms that are executed independently of each other is somewhat  related to writing the loop-free categories \({\vec\pi}_1(V)\) as finite product, for \(V\) partially ordered open subspace of \(\llbracket P\rrbracket\). 
We expect that these local decompositions actually induce a global decomposition of \({\vec\pi}_1(\llbracket P\rrbracket)\); our hope is based on the van Kampen theorem \cite[4.52, p.80]{datc}.  

\section{Definitions}
    \label{sec:definitions}

    Given an object~$X$ of a category $\Cat{C}$, we write $\id{X}$ for the identity
    morphism on $X$. We generally omit the subscript when clear from the
    context. Given a morphism $f:X\to Y$, we write $\dom{f}=X$ (\resp $\codom{f}=Y$)
    for its \emph{source} (\resp \emph{target}).
    \medskip
    
    \begin{definition}
      \label{def:return}
      A morphism $f\colon x \to y$ in a category $\mathcal{C}$ is said to be
      \emph{without return} when the hom-set $\mathcal{C}(y,x)$ is empty. Otherwise,
      we say that $f$ \emph{admits a return}.
      A category $\mathcal{C}$ is said to be \emph{loop-free} when all its
      morphisms, except identities, are without return.
      We write $\lfcat$ for the category of all small loop-free categories.
    \end{definition}
    \medskip
    
    \begin{definition}
      In $\lfcat$, consider a family
      $\set{f_\alpha : X_\alpha \to Y_\alpha}_{\alpha \in A}$ of morphisms. The
      \emph{product map}
      $f : \prod_{\alpha \in A}X_\alpha \to \prod_{\alpha \in A}Y_\alpha$ is the
      unique morphism such that $\pi_\alpha f = f_\alpha$ for
      every $\alpha\in A$.
    \end{definition}
    \medskip
    
    For the rest of the chapter we will often write $f = (f_\alpha)_{\alpha \in A}$
    for the product map of the family $\set{f_\alpha : X_\alpha \to Y_\alpha}_{\alpha \in A}$.
    For binary products, we write $(f,g)$ the product map of $f$ and $g$.
    
    \begin{definition}
      A \emph{fence} is a family $(f_i)_{1\leq i \leq n}$ of morphisms of $\Cat{C}$ such that:
      \begin{itemize}
          \item $\codom{f_{2i}} = \codom{f_{2i+1}}$ and $\dom{f_{2i-1}} = \dom{f_{2i}}$
          \item or $\dom{f_{2i}} = \dom{f_{2i+1}}$ and $\codom{f_{2i-1}} = \codom{f_{2i}}$
      \end{itemize}
      The \emph{length} of a fence $(f_i)_{1\leq i \leq n}$ is its cardinal $n$
    \end{definition}
    \medskip
    
    \begin{definition}
      Two morphisms $f$ and $g$ of a category $\Cat{C}$ are said to be \emph{connected} 
      if there exists a fence $(f_i)_{i\in [1,n]}$ such that $f=f_1$ and \(f_n=g\).
      A category $\Cat{C}$ is said to be connected if all its morphisms are connected.
    \end{definition}
    \medskip

    \begin{definition}
      A \emph{(product) decomposition}
      of a loop-free category $\Cat{C}$ is an isomorphism $\Psi$ from $\Cat{C}$ to a 
      product $\prod_{\alpha \in A} X_\alpha$ of loop-free categories $X_\alpha$ for $\alpha \in A$.
    \end{definition}
    \medskip
    
    Note that if $\Cat{C}$ is non-empty, connected, and loop-free, then so are the categories $X_\alpha$.

\section{Hashimoto's theorem for loop-free categories}

Hashimoto's Theorem \cite{hashimoto1951direct} states that every connected partial order
has the strict refinement property; in this article, we generalize this result to all 
connected loop-free categories in \cref{thm:big-result}.


As previously stated, we will follow the proof from \cite{schroder2003ordered}, generalizing
the key lemmas to the case of connected loop-free categories.
We will first give a rough idea of the different steps of the proofs, before introducing the details and technical
lemmas necesssary for the proof in \cref{sec:technical}.

To prove \cref{thm:big-result} we will need to find the decompositions 
$(a_\alpha)_{\alpha \in A} $ and $(b_\beta)_{\beta \in B}$ of \cref{diag:hashi-global}.
For this we will first define some notions, that will be used in the presentation of the proof.
\medskip

\begin{definition}
    \label{def:slice}
    Given an element \(s\) of a set product 
    \(\prod_{i\in I}S_i\), an index \(j\in I\), and an element 
    \(x_j\in S_j\), we define \((s,x_j,j)_{i\in I}\) as the element of \(\prod_{i\in I}S_i\) obtained by substituting \(x_j\) to \(s_j\) in \(s\); in other words:
        \[(s,x_j,j)_i = \begin{cases}
            x_j \quad \text{if } i = j \\
            s_i \quad \text{otherwise.}
        \end{cases}\]
\end{definition}
    \medskip

\begin{definition}
    \label{def:section}
    Given an object $s$ of a loop-free category 
    $\prod_{\alpha \in A}X_\alpha$, and an index $\lambda \in A$, 
    the \emph{\(\lambda\)-section at $s$} is the functor from \(X_\lambda\) to \(\prod_{\alpha \in A}X_\alpha\) defined by 
    \[
    \left\{
    \begin{array}{lcll}
             x_\lambda & \mapsto & (s,x_\lambda,\lambda) & \text{if \(x_\lambda\) is an object} \\[1mm]
             f_\lambda & \mapsto & (\id{s},f_\lambda,\lambda) & \text{if \(f_\lambda\) is a morphism} 
        \end{array}\right.  \]
    \color{black}
    We denote by \(\Xi_\lambda^s:X_\lambda\cong X^s_\lambda\) the isomorphism induced by the \(\lambda\)-section at \(s\) on its image. 
    We denote by \(\Phi^s_\lambda\) the restriction of \(\Phi\) to \(X^s_\lambda\) 
    for every functor \(\Phi\) defined over \(\prod_{\alpha \in A}X_\alpha\). 
\end{definition}
    \medskip

\begin{definition}
    \label{def:alpha-beta}
    Given an isomorphism $\Phi \colon  \prod_{\alpha \in A } X_\alpha \to \prod_{\beta \in B} Y_\beta$, 
    indexes $\lambda \in A$, $\mu \in B$, and an object $s \in \prod_{\alpha \in A} X_\alpha$, we define
        \[X_{\lambda}^{\mu} \coloneqq \pi_{\lambda}\Phi^{-1}\Xi^{\Phi(s)}_{\mu}(Y_{\mu}) 
        \quad \text{and} \quad Y_{\mu}^{\lambda} \coloneqq \pi_{\mu}\Phi\Xi^s_{\lambda}(X_{\lambda})\]
    with \(\pi_{\lambda}:\prod_{\alpha \in A } X_\alpha \to X_{\lambda}\) and 
     \(\pi_{\mu}:\prod_{\beta \in A } X_\beta \to X_{\mu}\) the projections.
\end{definition}
    \medskip

As we will see, the refinement (but not its existence) 
\(Z_{\alpha,\beta}\) (with \(\alpha\in A\) and \(\beta\in B\)) depends on an object 
\(s\in\prod_{\alpha\in A}X_\alpha\) that we arbitrarily fix now if the product is non-empty.
In \cite{schroder2003ordered}, which proves the
strict refinement property for connected posets, the isomorphisms
\((a_\alpha)_{\alpha \in A} \) and \((b_\beta)_{\beta \in B}\)
of Diagram~\ref{diag:hashi-global} are obtained by decomposing the expected isomorphisms
into smaller morphisms, obtained with the notation introduced above, which gives the 
Diagram~\ref{diag:hashimoto-detailled} below.

\begin{equation}
    \begin{tikzcd}[]
        \label{diag:hashimoto-detailled}
        {\prod_{\alpha \in A}X_\alpha} &&& {\prod_{\beta \in B}Y_\beta} \\
        \\
        {\prod_{\alpha \in A}X_\alpha^s} &&& {\prod_{\beta\in B}Y_\beta^{\phi(s)}} \\
        \\
        {\prod_{\alpha \in A}\prod_{\beta\in B}Y_\beta^\alpha} \\
        \\
        {\prod_{\alpha\in A}\prod_{\beta\in B}X^\beta_\alpha} &&& {\prod_{\beta \in B}\prod_{\alpha \in A}X_\alpha^\beta}
        \arrow["\Phi", from=1-1, to=1-4]
        \arrow["{(\Xi_\alpha^s)_{\alpha\in A}}"', from=1-1, to=3-1]
        \arrow["{(\Phi|_{X^s_\alpha})_{\alpha\in A}}"', from=3-1, to=5-1]
        \arrow["{((\pi_\alpha \Phi^{-1} \Xi_\beta^{\Phi(s)})_{\beta \in B})_{\alpha \in A}}"', from=5-1, to=7-1]
        \arrow["{(\Xi_\beta^{\Phi(s)})_{\beta \in B}}", from=1-4, to=3-4]
        \arrow["{(\Phi^{-1}|_{X^{\Phi(s)}_\beta})_{\beta\in B}}", from=3-4, to=7-4]
        \arrow["\gamma^{-1}"', from=7-4, to=7-1]
    \end{tikzcd}
\end{equation}

The first part of the proof is to prove that all the introduced morphisms, are in fact isomorphisms
and to prove that they are indeed product maps of isomorphisms.
That is to say that we have 
\(\Psi_2 = \underset{\alpha\in A}\prod a_\alpha\) and \(\Psi_1 = \underset{\beta\in B}\prod b_\beta\) with 
\[
    a_\alpha = \bigl(\pi_\alpha \Phi^{-1}\Xi^{\Phi(s)}_\beta
    \pi_\beta\Phi \Xi^s_\alpha\bigr)_{\beta \in B}
     \qquad\text{and}\qquad
    b_\beta = \Phi^{-1} \Xi_\beta^{\Phi(s)}
\]

Most of the proofs rely heavily on the following proposition, which allows us
to restrain the coordinates of composite images.
\medskip

\restate{lem:proj-cont}

\medskip

Then, once this has been achieved, we need to prove that the diagram commutes to conclude the proof.
To do this, we will first prove the commutation on a restriction of the diagram,
by replacing for each \(\alpha\in A\) 
the starting category $X_\alpha$ by the subcategory \(X_\alpha^\mu\) as in \cref{def:alpha-beta},
for an arbitrary $\mu \in \beta$, giving us the diagram below. 

\begin{equation}
    \label{diag:hashi-small}
    \begin{tikzcd}[]
        {\prod_{\alpha \in A}X_\alpha^\mu} &&& {Y_\mu^{\phi(s)}} \\
        \\
        \\
        {\prod_{\alpha\in A}\prod_{\beta\in B}X^\beta_\alpha} &&& {\prod_{\beta \in B}\prod_{\alpha \in A}X_\alpha^\beta}
        \arrow["\Phi", from=1-1, to=1-4]
        \arrow["{\gamma^{-1}}"', from=4-4, to=4-1]
        \arrow["{\Psi_2}"', from=1-1, to=4-1]
        \arrow["{\Psi_1}", from=1-4, to=4-4]
    \end{tikzcd}
\end{equation}


Then using the following \cref{lem:ext-iso},
the commutation is extended along a given $\lambda \in A$.

    

\medskip

\restate{lem:ext-iso}

\medskip

This leads to the commutation of the 
Diagram~\ref{diag:hashi-al-end} below.

\begin{center}
    \begin{equation}
        \label{diag:hashi-al-end}
        \begin{tikzcd}[]
            {\prod_{\alpha\in A\setminus{\lambda}} X^\mu_\alpha \times X_\lambda} &&& {\prod_{\beta \in B} Y_\beta} \\
            \\
            \\
            {\prod_{\alpha\in A}\prod_{\beta\in B} X^\beta_\alpha} &&& {\prod_{\beta \in B}\prod_{\alpha \in A}X_\alpha^\beta}
            \arrow["\Phi", from=1-1, to=1-4]
            \arrow["{\gamma^{-1}}"', from=4-4, to=4-1]
            \arrow["{(\pi_\alpha\Psi_2)_{\alpha\neq \lambda}\times \pi_\lambda\Psi_2 }"', from=1-1, to=4-1]
            \arrow["{\Psi_1 }", from=1-4, to=4-4]
        \end{tikzcd}
    \end{equation}
\end{center}

Using \cref{lem:ext-iso} once more the commutativity can be extended
to the full domain $\prod_{\alpha \in A}X_\alpha$, thus ending the proof.

In our case a few necessary conditions that are less trivial will need to be detailed in \cref{lem:faith-restriction}
to perform the last steps of the proof, but the broad strokes will remain the same.

\section{Preliminary results}
    \label{sec:technical}
    As announced the proof of \cref{thm:big-result} relies on a few technical results and
    properties of loop-free categories that we will first present in this section.
    The \cref{cor:proj-cont} and \cref{lem:ext-iso} are especially important pieces of the inner working of the
    proof.

\subsection{Properties of loop-free categories}

We remind the reader of a few properties of loop-free categories that we will use in the following proofs, 
as well as two usefull properties specific to loop-free categories.

\begin{proposition}
    $\lfcat$ has all products.
\end{proposition}
  
  \begin{proof}
    $\lfcat$ is an epireflective subcategory of $\cat$ (\cite[Proposition 1.8]{haucourt2006categories}),
    a cartesian closed category. Thus, it has all products.
  \end{proof}
  
  By the adjunction between $\lfcat$ and $\poset$, we have the following result:
  
  \begin{proposition}
    \label{prop:func-to-iso}
    Let $\Phi \colon \Cat{C} \to \Dat$ be a morphism of connected loop-free
    categories. Then the restriction
    \[
      \Phi|_\mathsf{Obj} \colon (\mathsf{Obj}(\Cat{C}),\leq) \to (\mathsf{Obj}(\Dat),\leq)
    \]
    is an order-preserving morphism of connected posets
    with the canonical partial order $\leq$ defined by
      \[ X\leq Y \iff \Cat{C}(X,Y) \neq  \emptyset\]
    Furthermore, if $\Phi$ is
    an isomorphism, so is $\Phi|_\mathsf{Obj}$.
  \end{proposition}
  
  \begin{lemma}
    \label{lem:tri-id}
    Given two morphisms $f,g$ of a loop-free category $\Cat{C}$,
    we have that $f\circ g = \ident{}$ implies 
    $f = g = \ident{}$
  \end{lemma}
  
  \begin{proof}
      If $f\circ g = \ident{}$, then both \(f\) and \(g\) 
      have a return; it follows that both are identities because $\Cat C$ is loop-free.
  \end{proof}

  \begin{lemma}
    \label{lem:id-decomposition}
    Given an isomorphism of small, loop-free categories 
    $\Phi \colon \Cat{C}_1 \times \Cat{C}_2 \to \Cat{D}_1 \times \Cat{D}_2$ and
    an object $X$ of $\Cat{C}_1$, for all morphisms $(f,g)$ of $\Cat{D}_1\times \Cat{D}_2$ such that
    \(\pi_1\Phi^{-1}(f,g) = \id{X}\), we have:
    \[
      \pi_1\Phi^{-1}(f,\id{\codom{g}}) = \pi_1\Phi^{-1}(f,\id{\dom{g}}) = 
      \pi_1\Phi^{-1}(\id{\codom{f}},g) =  \pi_1\Phi^{-1}(\id{\dom{f}},g) = \id{X}.
    \]
  \end{lemma}
  
  \begin{proof}
      By definition of the product, the following diagram on the left commutes:
    \begin{center}
      \begin{tikzcd}[]
        \bullet \ar[rr, "\text{$(\ident{},g) $}"]   &  & \bullet                      \\
        {}\\
        \bullet \ar[uu, "\text{$(f,\ident{}) $}"] \ar[rr, "\text{$(\ident{},g)$}"'] \ar[rruu, "\text{$(f,g)$}"{description}] 
        & & \bullet \ar[uu, "\text{$(f,\ident{})$}"']
      \end{tikzcd}
      \scalebox{1.2}{$\overset{\pi_\Cat{P}\circ\Phi^{-1}}{\implies}$}
      \begin{tikzcd}[]
        {\bullet}\ar[rr, "{\pi_1\Phi^{-1}(\ident{},g) }"]  &  & {\bullet}                     \\
        {} \\
        {\bullet}\ar[uu, "{\pi_1\Phi^{-1}(f,\ident{}) }"] \arrow[rr, "{\pi_1\Phi^{-1}(\ident{},g)}"'] 
        \arrow[rruu, "{id}"{description}] &  & {\bullet}\ar[uu, "{\pi_1\Phi^{-1}(f,\ident{})}"']
      \end{tikzcd}
    \end{center}
  
      The category $\Cat P$ is loop-free, thus by applying \cref{lem:tri-id} we have  
      \[\pi_1\Phi^{-1}(\id{\dom{f}},g) = 
      \pi_1\Phi^{-1}(f,\id{\codom{g}}) = \id{}
      \]
      We prove in the same way that   
      \(\pi_1\Phi^{-1}(\id{\codom{f}},g) = 
      \pi_1\Phi^{-1}(f,\id{\dom{g}}) = \id{}\).
  \end{proof}

\subsection{Technical lemmas}

The proof of Hashimoto's theorem in \cite{schroder2003ordered} makes extensive use of the  
\cref{prop::proj-cont} and
\cref{prop:id-sliding},
which we will respectively extend to connected elements of \(\lfcat\) in 
\cref{lem:proj-cont} and
\cref{lem:ext-iso}.
As we have changed the formulation to make the proofs easier to follow, we give them here and refer to the original work
for the proof.

These lemmas are at the core of the proof and are where the hypothesis that we are using loop-free
and connected categories really comes into play, so it is important to keep them in mind.

\medskip

\begin{proposition} 
    \label{prop::proj-cont}
    Given
    \(\Phi \colon P \times Q \to U \times V\) an isomorphism of connected posets and $p \in P$,
    \[ \pi_\Cat{P}\Phi^{-1}(u,v) = \pi_\Cat{P}\Phi^{-1}(u',v') = p  \text{ implies }
    \pi_\Cat{P}\Phi^{-1}(u,v') = p \]
\end{proposition}
\begin{proof}
    \cite[Lemma 10.4.5]{schroder2003ordered}.
\end{proof}

\begin{proposition}
    \label{prop:id-sliding}
    Given $\Phi \colon P\times Q \to U \times V$ an isomorphism of connected posets,
    if there exists $p\in P$ such that \(\pi_U\Phi(p,q) = \pi_U \Phi(p,q')\), then for
    each $p'\in  P$, \[\pi_U\Phi(p',q) = \pi_U \Phi(p',q')\]
\end{proposition}
\begin{proof}
    \cite[Lemma 10.4.8]{schroder2003ordered}
\end{proof}

In the two following proofs, as the objects in the commutative diagrams are of no importance, we have omitted them,
replacing them by $\bullet$ when not necessary.

\begin{restatethis}{lemma}{lem:proj-cont}
    \label{lem:proj-cont}
    Let $\Phi \colon \Cat{P} \times \Cat{Q} \to \Cat{U} \times \Cat{V}$ an isomorphism of connected loop-free categories.
    Let $(u,v)$ and $(u',v')$ two morphisms of $\Cat{U} \times \Cat{V}$ and $p$ a morphism of $\Cat{P}$. Then $\pi_\Cat{P}\Phi^{-1}(u,v) = \pi_\Cat{P}\Phi^{-1}(u',v') = p$ implies
    $\pi_\Cat{P}\Phi^{-1}(u,v') = p$.
\end{restatethis}

\begin{proof}
    Given $p,q,q'$ such that $\Phi(p,q) = (u,v)$ and $\Phi(p,q') = (u',v')$. Let us prove that $\pi_\Cat{P}\Phi^{-1}(u,v') = p$.
    We proceed by induction on the length of the fence between $q$ and $q'$.
    \begin{itemize}
        \item A fence of length $2$ implies $\codom{q}= \codom{q'}$ or $\dom{q}= \dom{q'}$.
            First, let us suppose $\codom{q}= \codom{q'}$, the other case being solved dually.\\
            Let us prove that $\pi_\Cat{P}\Phi^{-1}(u,v') = p$.\\
            First, let us define
            \begin{align}
                \label{def:2fence-img}
                \Phi(p,\id{\codom{q}}) := (u^*,v^*) &  & \Phi(\id{\dom{p}},q) := (u_q,v_q) &  & \Phi(\id{\dom{p}},q') := (u_q',v_q')
            \end{align}
            such that:
              \begin{align*}
                  u= u_q\circ u^*   &  & u' = u_q' \circ u^* \\
                  v = v_q \circ v^* &  & v' = v_q' \circ v^*
              \end{align*}
              By the \cref{prop:bin-prod}, the following diagram commutes:\\
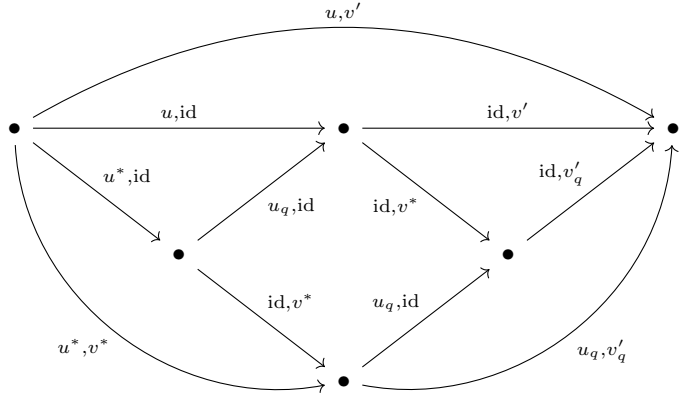
              \begin{center}
                  \begin{tikzcd}[]
                      {\bullet} \arrow[rrdd, "{u^*,\ident}"] \arrow[rrrr, "{u,\ident}"] \arrow[rrrrrrrr, "{u,v'}", bend left] \arrow[rrrrdddd, "{u^*,v^*}"', bend right=49] &  &                                                     &  & {\bullet} \arrow[rrdd, "{\ident,v^*}"'] \arrow[rrrr, "{\ident,v'}"]                        &  &                             &  & {\bullet} \\
                      &  &                                                     &  &                                                                           &  &                             &  &    \\
                      &  & {\bullet} \arrow[rruu, "{u_q,\ident}"'] \arrow[rrdd, "{\ident,v^*}"] &  &                                                                           &  & {\bullet} \arrow[rruu, "{\ident,v_q'}"] &  &    \\
                      &  &                                                     &  &                                                                           &  &                             &  &    \\
                      &  &                                                     &  & {\bullet} \arrow[rruu, "{u_q,\ident}"] \arrow[rrrruuuu, "{u_q,v_q'}"', bend right=49] &  &                             &  &
                  \end{tikzcd}
              \end{center}

              Following the outer arrows, we get $$(u,v') = (u_q,v_q') \circ (u^*,v^*)$$\\

              By \cref{def:2fence-img}, $\pi_\Cat{P}\Phi^{-1}(u^*,v^*) = p$.
              We are thus left to prove $\pi_\Cat{P}\Phi^{-1}(u_q,v_q') = \ident$.
              By construction $\codom{v_q} = \codom{v_q'} = \dom{v^*}$ such that by functoriality:
              \begin{align*}
                  \codom{\pi_\Cat{P}\Phi^{-1}(u_q,v_q')} & = \pi_\Cat{P}\Phi^{-1}(\codom{u_q},\codom{v_q'}) \\
                                                         & =  \pi_\Cat{P}\Phi^{-1}(\codom{u_q},\codom{v_q}) \\
                                                         & = \codom{\pi_\Cat{P}\Phi^{-1}(u_q,v_q)}          \\
                  \codom{\pi_\Cat{P}\Phi^{-1}(u_q,v_q')} & = \dom{p}
              \end{align*}

              By \cref{prop:func-to-iso}, $\Phi \colon \Obj{\Cat{P}} \times \Obj{\Cat{Q}} \to \Obj{\Cat{U}} \times \Obj{\Cat{V}}$ is an order-preserving isomorphism of connected posets and such that
              $\pi_\Cat{P}\Phi^{-1}(\dom{u_q},\dom{v_q}) = \pi_\Cat{P}\Phi^{-1}(\dom{u_q'},\dom{v_q'}) = \dom{p}$. Hence,
              \begin{align*}
                  \dom{p} & = \pi_\Cat{P}\Phi^{-1}(\dom{u_q},\dom{v_q'}) &\text{\cref{prop::proj-cont} } \\
                  \dom{p} & = \dom{\pi_\Cat{P}\Phi^{-1}(u_q,v_q')} &\text{by functoriality}
              \end{align*}

              Thus, $\pi_\Cat{P}\Phi^{-1}(u_q, v_q') \in \Cat{P}(\dom{p},\dom{p})$, with $\Cat{P}$ loop-free. This implies $\pi_\Cat{P}\Phi^{-1}(u_q, v_q') = \id{\dom{p}}$, such that
              \begin{align*}
                  \pi_\Cat{P}\Phi^{-1}(u,v') & = \pi_\Cat{P}\Phi^{-1}(u_q,v_q')  \circ  \pi_\Cat{P}\Phi^{-1}(u^*,v^*) \\
                  \pi_\Cat{P}\Phi^{-1}(u,v') & = p
              \end{align*}

        \item 
        Now let us suppose a fence
        \begin{tikzcd}
            {\cdot} \ar[r, "q'"] &[-12pt] {\cdot}  &[-12pt] {\cdot} \ar[l, "r"']  \ar[r, "q"] &[-12pt] {\cdot}
          \end{tikzcd}
        of length $n=3$ between $q$ and $q'$.
              Let
              \begin{align}
                  \label{def:3fence-img}
                  (u,v) := \Phi(p,q)
                   &  & (x,y) := \Phi(p,r) &  & (u',v') := \Phi(p,q')
              \end{align}
              such that $\Phi$ sends the commutative diagram on top to the one below.
              \begin{center}
                  \begin{tikzcd}[]
                      {\bullet} \arrow[rr, "{\ident,q'}"] \arrow[rrdd, "{p,q'}"'] &  & {\bullet} \arrow[dd, "{p,\ident}"] &  & {\bullet} \arrow[ll, "{\ident,r}"']
                      \arrow[lldd, "{p,r}"] \arrow[dd, "{p,\ident}"] \arrow[rrdd, "{p,q}"] &  &    \\
                      &  &                         &  &                                                                                              &  &    \\
                      &  & {\bullet}                      &  & {\bullet} \arrow[ll, "{\ident,r}"] \arrow[rr, "{\ident,q}"']                                                &  & {\bullet}
                  \end{tikzcd}
                  \\[7pt]
                  \begin{align}
                    \label{eq:3fence}
                    \begin{tikzpicture}
                        \draw (0,0) node[rotate=-90] {$\Longrightarrow$};
                        \draw (0,0) node[right] {$\Phi$};
                    \end{tikzpicture}
                  \end{align}
                  \\[7pt]
                  \begin{tikzcd}[]
                      {\bullet} \arrow[rr, "{u_q',v_q'}"] \arrow[rrdd, "{u',v'}"'] &  & {\bullet} \arrow[dd, "{f^*,y^*}"] &  & {\bullet} \arrow[ll, "{u_r',v_r'}"']
                      \arrow[lldd, "{x,y}"] \arrow[dd, "{u^*,v^*}"] \arrow[rrdd, "{u,v}"] &  &    \\
                      &  &                         &  &                                                                                              &  &    \\
                      &  & {\bullet}                      &  & {\bullet} \arrow[ll, "{u_r,v_r}"] \arrow[rr, "{u_q,v_q}"']                                                &  & {\bullet}
                  \end{tikzcd}
              \end{center}
              We are going to use the same method as before, working on each 2-fence inside the 3-fence above.
              For that we'll decompose $(u,v')$ using the Diagram \ref{eq:3fence}.
              By the \cref{prop:bin-prod}, the following diagram commutes:\\
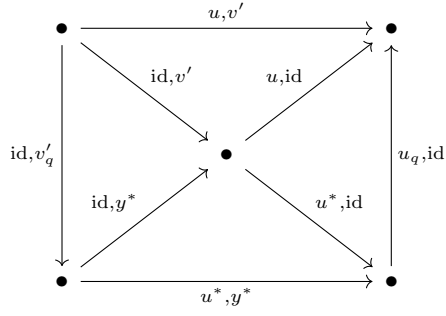
              \begin{center}
                  \begin{tikzcd}[]
                      {\bullet} \arrow[rrdd, "{\ident,v'}"] \arrow[dddd, "{\ident,v_q'}"'] \arrow[rrrr, "{u,v'}"] &  &                                                  &  & {\bullet}                                \\
                      &  &                                                  &  &                                   \\
                      &  & {\bullet} \arrow[rruu, "{u,\ident}"] \arrow[rrdd, "{u^*,\ident}"] &  &                                   \\
                      &  &                                                  &  &                                   \\
                      {\bullet} \arrow[rruu, "{\ident,y^*}"] \arrow[rrrr, "{u^*,y^*}"']                           &  &                                                  &  & {\bullet} \arrow[uuuu, "{u_q,\ident}"']
                  \end{tikzcd}
              \end{center}
              such that
              \begin{align*}
                  (u,v') & = (u_q,\id{\dom{v'}}) \circ (u^*,y^*) \circ (\id{\codom{u^*}},v_q')   &  & \text{By following the outer arrows}  \\
                  (u,v') & = (u_q,\id{\dom{v_r}}) \circ (u^*,y^*) \circ (\id{\codom{u_r'}},v_q') &  & \text{By the Diagram \ref{eq:3fence}}
              \end{align*}

              \begin{itemize}
                  \item Working on the 2-fence $\overset{q}{\gets}\cdot\overset{r}{\to}$, let us prove $\pi_\Cat{P} \Phi^{-1}(u_q,\id{\dom{v_r}}) = \id{}$ \\
                        By \cref{def:3fence-img} we have
                        $\pi_\Cat{P}\Phi^{-1}(u,v) = \pi_\Cat{P}\Phi^{-1}(x,y) = p$. Thus, as proved for 2-fences above
                        \begin{align*}
                            p & = \pi_\Cat{P}\Phi^{-1}(u,y)                                            &  &                                \\
                              & = \pi_\Cat{P}\Phi^{-1}((u_q,v_r)\circ(u^*,v^*))                        &  & \text{Diagram \ref{eq:3fence}} \\
                              & =   \pi_\Cat{P}\Phi^{-1}(u_q,v_r) \circ  \pi_\Cat{P}\Phi^{-1}(u^*,v^*) &  & \text{Functoriality}           \\
                            p & = \pi_\Cat{P}\Phi^{-1}(u_q,v_r) \circ p                                &  & \text{Diagram \ref{eq:3fence}}
                        \end{align*}
                        By loop-free property of $\Cat{P}$, $\pi_\Cat{P}\Phi^{-1}(u_q,v_r) = \id{\dom{p}}$.
                        Which implies by \cref{lem:id-decomposition}, 
                        \[\pi_\Cat{P} \Phi^{-1}(u_q,\id{\dom{v_r}}) = \id{\dom{p}}\]
                        Similarly, $\pi_\Cat{P} \Phi^{-1}(\id{\codom{u_r'}},v_q') = \id{\codom{p}}$.
                        Hence
                        \begin{align*}
                            \pi_\Cat{P} \Phi^{-1}(u,v') & = \pi_\Cat{P} \Phi^{-1}(u_q,\id{\dom{v_r}}) \circ \pi_\Cat{P} \Phi^{-1}(u^*,y^*) \circ \pi_\Cat{P} \Phi^{-1}(\id{\codom{u_r'}},v_q') \\
                                                        & = \id{\dom{p}} \circ \pi_\Cat{P} \Phi^{-1}(u^*,y^*) \circ \id{\codom{p}}                                                             \\
                            \pi_\Cat{P} \Phi^{-1}(u,v') & = \pi_\Cat{P} \Phi^{-1}(u^*,y^*)
                        \end{align*}

                  \item Let us prove now $\pi_\Cat{P} \Phi^{-1}(u^*,y^*) = p$.
                        By commutativity of both projections the central square:
                        \begin{align*}
                            (x,y) & = (u_r\circ u^*, y^* \circ v_r')                                        \\
                                  & = (u_r,\id{\dom{u^*}}) \circ (u^*,y^*) \circ (\id{\codom{u^*}},v_r')    \\
                            (x,y) & = (u_r,\id{\codom{v_r}}) \circ (u^*,y^*) \circ (\id{\codom{u_r'}},v_r')
                        \end{align*}
                        As $\pi_\Cat{P}\Phi^{-1}(u_r,v_r) = \id{\dom{p}}$, by \cref{lem:id-decomposition}, $\pi_\Cat{P} \Phi^{-1}(u_r,\id{\dom{v_r}}) = \id{\dom{p}}$.
                        Similarly, we have  \(\pi_\Cat{P} \Phi^{-1}(\id{\codom{u_r'}},v_r') = \id{\codom{p}}\). Thus
                        \[p = \pi_\Cat{P} \Phi^{-1}(x,y) = \pi_\Cat{P} \Phi^{-1}(u^*,y^*)\]
              \end{itemize}

              Such that
              \begin{align*}
                  \pi_\Cat{P} \Phi^{-1}(u,v') & = p
              \end{align*}
        \item 
        Now let us suppose that the property holds for all fences of length $1 \leq k < n$, with $n>3$. Let us suppose given
        $q,q'$, connected by a fence $(q_i)_{1\leq i \leq n}$.
        We define $(u_i,v_i) \coloneqq  \Phi(p,q_i)$, for each $1 \leq i \leq n$ and 
        \begin{align*}
            (u,v) &\coloneqq (u_1,v_1) = \Phi(p,q_1) = \Phi(p,q) \\
        (u',v') &\coloneqq (u_n,v_n) = \Phi(p,q_n) = \Phi(p,q')
        \end{align*}
        By definition, $((p,q_i))_{1 \leq i \leq n-1}$ and $((p,q_i))_{2 \leq i \leq n}$ are fences of length $n-1$ respectively connecting
        $(p,q_1)$, $(p,q_{n-1})$ and $(p,q_2)$, $(p,q_{n})$. By definition of $(u_i,v_i)$ we have,
        \begin{align*}
            p & = \pi_\Cat{P}\Phi^{-1}(u_1,v_1) = \pi_\Cat{P}\Phi^{-1}(u_{n-1},v_{n-1}) \\
            p & = \pi_\Cat{P}\Phi^{-1}(u_2,v_2) = \pi_\Cat{P}\Phi^{-1}(u_{n},v_{n})
        \end{align*}
        Thus, by induction hypothesis, this implies
        \begin{align}
            \label{eq:bite}
          p&=\pi_\Cat{P}\Phi^{-1}(u_1,v_{n-1})  &  p=& \pi_\Cat{P}\Phi^{-1}(u_2,v_{n})
        \end{align}
        Now we suppose $\dom{q_1} = \dom{q_2}$ ($\codom{q_1} = \codom{q_2}$ can be treated dually). 
        We'll have to proceed differently depending on the symmetry of the fence.
        \begin{itemize}
            \item If $\dom{q_{n-1}} = \dom{q_n}$, by functoriality of $\Phi$, $\dom{u_{n-1}},\dom{v_{n-1}} = \dom{u_n},\dom{v_n}$, such that
                  $\dom{u_2},\dom{v_n} = \dom{u_1},\dom{v_{n-1}}$.
                  This implies that $(u_1,v_{n-1})$ and $(u_2,v_n)$ are connected by a fence of length $2$.
                  Thus, $\pi_{\Cat{Q}}\Phi^{-1}(u_1,v_{n-1})$ and $\pi_{\Cat{Q}}\Phi^{-1}(u_2,v_n)$ are connected by a fence of length $2$.
                  By \cref{eq:bite}, we have $p = \pi_\Cat{P}\Phi^{-1}(u_2,v_n)
                  = \pi_\Cat{P}\Phi^{-1}(u_1,v_{n-1})$.
                  We can then apply our induction hypothesis, such that
                  \begin{align*}
                      p & = \pi_\Cat{P}\Phi^{-1}(u_1,v_n)\\
                      p & = \pi_\Cat{P}\Phi^{-1}(u,v')
                  \end{align*}
                  Which concludes this case.
            \item If $\codom{q_{n-1}} = \codom{q_n}$,  by functoriality of $\Phi$, $\codom{u_{n-1}},\codom{v_{n-1}} = \codom{u_n},\codom{v_n}$.
            Thus, if we define $q_{2,n} = \pi_{\Cat{Q}}\Phi^{-1}(u_2,v_n)$ and $q_{1,n-1} =\pi_{\Cat{Q}}\Phi^{-1}(u_1,v_{n-1})$,
            we get the following fences:
            
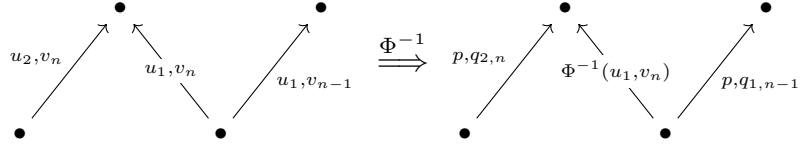
\begin{center}
                \begin{tikzcd}[]
                    & {\bullet} &                                                          & {\bullet} \\
                    &    &                                                          &   \\
                    {\bullet} \arrow[ruu, "{u_2,v_n}"] &    & {\bullet} \arrow[luu, "{u_1,v_n}"{description}] \arrow[ruu, "{u_1,v_{n-1}}"'] &
                \end{tikzcd}
                \scalebox{1.2}{$\overset{\Phi^{-1}}{\implies}$}
                \begin{tikzcd}[]
                    & {\bullet} &                                                          & {\bullet} \\
                    &    &                                                          &    \\
                    {\bullet} \arrow[ruu, "{p,q_{2,n}}"] &    & {\bullet} \arrow[luu, "{\Phi^{-1}(u_1,v_n)}"{description}] \arrow[ruu, "{p,q_{1,n-1}}"'] &
                \end{tikzcd}
            \end{center}
            Thus, $q_{2,n}$ et $q_{1,n-1}$ are connected by a fence of length $3$.
            We can apply our induction case for $k=3$
            with the fence $(p,q_{1,n-1})$, $(p,\pi_{\Cat{Q}}\Phi^{-1}(u_1,v_n))$, $(p,q_{2,n})$.
            Indeed, we have
            $\Phi(p,q_{1,n-1}) = (u_1,v_{n-1})$ and $\Phi(p,q_{2,n})= (u_2,v_n)$ (\cref{eq:bite}).
            Thus, we get
            \[\pi_\Cat{P}\Phi^{-1}(u_1,v_n) = p\]
        \end{itemize}
              Thus, $\pi_\Cat{P}\Phi^{-1}(u,v) = \pi_\Cat{P}\Phi^{-1}(u',v') = p$ implies  $\pi_\Cat{P}\Phi^{-1}(u,v') = p$, proving the induction step.
    \end{itemize}
\end{proof}

This \cref{lem:proj-cont} is not restricted to binary product 
and easily extends to arbitrary products as shown in the following lemma.
\begin{corollary}
    \label{cor:proj-cont}
    Let $\Phi \colon \prod_{\alpha \in A} X_\alpha \to \prod_{\beta \in B} Y_\beta$ an isomorphism of connected categories.
    Let $ f^i = (f^i_\alpha)_{\alpha \in A}$ a morphism of $\prod_{\alpha \in A} X_\alpha$ for $i=1,2$.
    Then for all $\lambda \in A$ and all $\mu \in B$ such that
    $\pi_\lambda f^1 = \pi_\lambda f^2$
    \[
        \pi_\lambda \Phi^{-1}(\Phi(f^1),\Phi(f^2)_\mu,\mu) = \pi_\lambda f^1 = \pi_\lambda f^2 \]
\end{corollary}

\begin{proof}
    This follows directly from \cref{lem:proj-cont} by defining:
    \begin{align*}
        P & \coloneqq X_\lambda                                                      &  &  & Q & \coloneqq \prod_{\alpha \in A,\alpha\neq\lambda} X_\alpha &  &  &
        U & \coloneqq  Y_\beta    &  &  &
        V & \coloneqq \prod_{\beta\in B,\mu\neq \beta} Y_\beta
    \end{align*}
    and by defining
    \begin{align*}
        (p,q) &\coloneqq \delta_\lambda^{-1}(f^1) \qquad& (p,q') &\coloneqq \delta_\lambda^{-1}(f^2) \\
        (u,v) &\coloneqq \delta_\mu^{-1}\Phi(f^1) \qquad& (u',v') &\coloneqq \delta_\mu^{-1}\Phi(f^2)
    \end{align*}
    With $\delta_\lambda$ and $\delta_\mu$ the natural isomorphisms:
    \begin{align*}
        \delta_\lambda \colon X_\lambda \times \prod_{\alpha \in A,\alpha\neq\lambda} X_\alpha
        \to  \prod_{\alpha \in A} X_\alpha \qquad
        \delta_\beta \colon 
        Y_\mu
        \times\prod_{\beta\in B,\mu\neq \beta} Y_\beta   \to \prod_{\beta \in B} Y_\beta
    \end{align*}
    We get an isomorphisms $\Psi = \delta_\beta^{-1} \circ \Phi \circ \delta_\lambda \colon P \times Q \to U \times V$.
    Such that $ f^i_\lambda = \pi_\Cat{P}\Psi^{-1}(u,v')$ for $i=1,2$.
\end{proof}

\begin{restatethis}{lemma}{lem:ext-iso}
    \label{lem:ext-iso}
    Let $\Phi \colon \Cat{P}\times \Cat{Q} \to \Cat{U} \times \Cat{V}$ be an isomorphism of connected
    loop-free categories.
    Let $\Cat{Q'}$ a connected subcategory of $\Cat{Q}$. Let $p$ a morphism of $\Cat{P}$.
    Then $\pi_U\Phi(p,q) = \pi_U\Phi(p,q')$ for all $q,q' \in \Cat{Q'}$ implies, for all $p' \in \Cat{P}$,
    $\pi_U\Phi(p',q) = \pi_U\Phi(p',q')$ for all $q,q' \in \Cat{Q'}$.
\end{restatethis}

\begin{proof}
    Let $p$ a morphism of $\Cat{P}$ Let $\Cat{Q'}$ a connected sub-category of $\Cat{Q}$, such that for all
    morphisms $q$, $q'$ in $\Cat{Q'}$, $\pi_U \Phi(p,q) = \pi_U\Phi(p,q')$.
    Let us prove that for all $p'$ morphism of $\Cat{P}$, $\pi_U\Phi(p',q') = \pi_U\Phi(p',q)$ 
    by induction on the length $n$ of a given fence connecting
    $q,q'$.

    \begin{itemize}
        \item $n=1$. Trivial.
        \item $n=2$.
              $\pi_U \Phi(p,q) = \pi_U\Phi(p,q')$ for all
              morphisms $q$, $q'$ in $\Cat{Q'}$ implies
              \begin{align*}
                  \codom{\pi_U \Phi(p,q)}         & = \codom{ \pi_U\Phi(p,q')}                                                    \\
                  \pi_U \Phi(\codom{p},\codom{q}) & = \pi_U\Phi(\codom{p},\codom{q'}) &  & \text{for all } q,q' \in \Mor{\Cat{Q}} \\
                  \pi_U \Phi(\codom{p},Q)         & = \pi_U\Phi(\codom{p},Q')         &  & \text{for all } Q,Q' \in \Obj{\Cat{Q}}
              \end{align*}
              By \cref{prop:func-to-iso} $\Phi|_\mathsf{Obj}$ is an isomorphism between $\Obj{\Cat{P}}\times\Obj{\Cat{Q}}$ and $\Obj{\Cat{U}}\times\Obj{\Cat{V}}$. By
              \cref{prop:id-sliding}
              \begin{align}
                  \label{eq:int-iso-sliding}
                  \text{for all }Q,Q' \in \Obj{\Cat{Q}},  \pi_U\Phi(\codom{p'},Q) = \pi_U\Phi(\codom{p'},Q')
              \end{align}
              Thus for $q\in \Mor{\Cat{Q'}}$ we have
              \begin{align*}
                  \codom{\pi_U\Phi(\id{\codom{p'}},q)} & = \pi_U\Phi(\codom{p'},\codom{q})                        \\
                                                       & = \pi_U\Phi(\codom{p'},\dom{q})            &  & \text{By
                  \cref{eq:int-iso-sliding}}                                                                      \\
                                                       & = \pi_U\Phi(\dom{\id{\codom{p'}}},\dom{q})               \\
                  \codom{\pi_U\Phi(\id{\codom{p'}},q)} & = \dom{\pi_U\Phi(\id{\codom{p'}},q)}
              \end{align*}
            
            As $\Cat{Q'}$ is loop-free, this implies $\pi_U\Phi(\id{\codom{p'}},q) = \ident$.
            Similarly, $\pi_U\Phi(\id{\codom{p'}},q') = \ident$.

              Let us suppose that $\dom{q} = \dom{q'}$, the other case being dual. By the above argument, the following diagrams commute
              \begin{center}

                  \begin{tikzcd}
                      {\bullet} \arrow[dd, "{(\id{\codom{p'}},q)}"'] \arrow[rrdd, "{(p',q)}"]   &[10pt]  &    \\
                      &  &    \\
                      {\bullet} \arrow[rr, "{(p',\id{\dom{q}})}"]                              &[10pt]  & {\bullet} \\
                      &  &    \\
                      {\bullet} \arrow[uu, "{(\id{\codom{p'}},q')}"] \arrow[rruu, "{(p',q')}"'] &[10pt]  &
                  \end{tikzcd}
                  \qquad
                  \scalebox{1.2}{$\overset{\pi_U\Phi}{\implies}$}
                  \qquad
                  \begin{tikzcd}
                      {\bullet} \arrow[dd, "{\id{}}"'] \arrow[rrdd, "{\pi_U\Phi(p',q)}"]   &[10pt]   &    \\
                      &  &    \\
                      {\bullet} \arrow[rr, "{\pi_U\Phi(p',\id{\dom{q}})}"]  &[10pt]  & {\bullet} \\
                      &  &    \\
                      {\bullet} \arrow[uu, "{\id{}}"] \arrow[rruu, "{\pi_U\Phi(p',q')}"'] &[10pt]   &
                  \end{tikzcd}
              \end{center}

              Therefore, $\pi_U\Phi(p',q') = \pi_U\Phi(p',\id{\dom{q}}) = \pi_U\Phi(p',q)$.
        \item Now suppose a fence $(q=q_0,q_1 \cdots, q_n = q')$ and the property true for all integers strictly smaller than $n$.
              Then there is a $n-1$-fence $(q_1, \cdots, q_n)$ and a $1$-fence $(q_0,q_1)$. By induction hypothesis,
              $\pi_U\Phi(p',q_n) = \pi_U\Phi(p',q_1) = \pi_U\Phi(p',q_0) $.
    \end{itemize}
    $\Cat{Q'}$ is connected so for all $p'\in \Cat{P}$, for all $q,q' \in \Cat{Q'}$, $\pi_U\Phi(p',q') = \pi_U\Phi(p',q')$.
\end{proof}

\section{Proof of \cref{thm:big-result}}

For the remainder of section, we will consider that $\Phi$ is an isomorphism of connected loop-free categories,
$(X_\alpha)_{\alpha\in A}$ and $(Y_\beta)_{\beta \in B}$ families of connected loop-free categories.

As stated before, we will first prove that all the morphisms of the Diagram
\ref{diag:hashimoto-detailled} are isomorphisms.

\medskip

\begin{lemma}
    \label{lem:iso-slices}
    Given a category \(X = \prod_{\alpha\in A} X_\alpha\), \(\lambda \in A\), $s$ a morphism of
    \(X\) and
    $\Xi^s_\lambda$ as defined in \cref{def:section},
    then
    \[\pi_\lambda \circ \Xi^s_\lambda = \id{X_\lambda}\]
    We say that $\Xi^s_\lambda$ is a \emph{section} of the canonical projection
    \(\pi_\lambda \colon \prod_{\alpha\in A} X_\alpha \to X_\lambda\).
    Furthermore, $\Xi^s_\lambda$ is a full and faithful functor
\end{lemma}

\begin{proof}
    Let $f \in \prod_{\alpha \in A} X_\alpha((s,x_\lambda, \lambda),(s,y_\lambda, \lambda))$.
    Then $X_\alpha$ loop-free implies, \(\pi_\alpha f =
    \id{s_\alpha}\) if \(\alpha \neq \lambda\)
    and $\pi_\lambda f \coloneqq f_\lambda \in X_\lambda(x_\lambda,y_\lambda)$.
    Such that $f = \Xi^s_\lambda(f_\lambda)$.
    This proves that $\Xi^s_\lambda$ is full. Furthermore, it is clearly faithful.
\end{proof}

\begin{corollary}
    \label{cor:full-sub-cat}
    Given a category \(X = \prod_{\alpha \in A} X_\alpha\),  \(\lambda \in A\), $s$ a morphism of
    \(X\) then
    $X^s_\lambda = \Xi^s_\lambda(X)$ is a connected, loop-free
    full sub-category of $X$ isomorphic to $X_\lambda$.
    Furthermore \[ \pi_\lambda \circ \Xi^s_\lambda = \id{{X_\lambda}}  \quad  \Xi^s_\lambda \circ \pi_\lambda|_{X^s_\lambda}= \id{{X^s_\lambda}}\]
\end{corollary}

One last proposition that we will need from \cite{schroder2003ordered} 
is the fact that for any \(\lambda \in A\), $s_\lambda$ object of $X_\lambda$, 
the object part of the functor \(\Phi|_{X^s_\lambda} \colon   X_\lambda^s \to \prod_{\beta \in B} Y^\lambda_\beta\)
is an isomorphism of the underlying objects of the category.

\medskip

\begin{proposition}
    \label{prop:iso-restriction}
    \cite[Lemma 10.4.7]{schroder2003ordered}\\
    Let $\Phi \colon  \prod_{\alpha \in A } X_\alpha \to \prod_{\beta \in B} Y_\beta$ be an isomorphism of connected
    posets. Let $s \in \prod_{\alpha \in A} X_\alpha$ and $\lambda \in A$ and $X_\lambda^s$ be as in \cref{lem:iso-slices}.
    Let $Y^\lambda_\beta = \pi_\beta\Phi[X^s_\lambda]$, then 
    \begin{align*}
        \Phi|_{X^s_\lambda} \colon   X_\lambda^s \to \prod_{\beta \in B} Y^\lambda_\beta
    \end{align*}
    is an isomorphism of posets.
\end{proposition}

This proposition also translates to an isomorphism of connected categories.

\begin{proposition}
    \label{lem:iso-restriction}
    Let $\Phi \colon  \prod_{\alpha \in A } X_\alpha \to \prod_{\beta \in B} Y_\beta$ be an isomorphism of connected
    loop-free categories. Let $s \in \prod_{\alpha \in A} X_\alpha$ and $\lambda \in A$ and $X_\lambda^s$ be as in \cref{lem:iso-slices}.
    Let $Y^\lambda_\beta = \pi_\beta\Phi[X^s_\lambda]$, then 
    \[\Phi|_{X^s_\lambda} \colon   X_\lambda^s \to \prod_{\beta \in B} Y^\lambda_\beta\]
    is an isomorphism of connected loop-free categories.
\end{proposition}

\begin{proof}
    Let us show that $\Phi|_{X^s_\lambda} \colon   X_\lambda^s \to \prod_{\beta \in B} Y^\lambda_\beta$ is essentially surjective.
    By \cref{prop:func-to-iso}, $\Phi_\mathsf{Obj}|_{X^s_\lambda}$ is an isomorphism of posets. Thus, by \cref{prop:iso-restriction},
    \begin{align*}
        \Phi_\mathsf{Obj}[X^s_\lambda] & = \prod_{\beta \in B} \pi_\beta\Phi_\mathsf{Obj}[\Obj{X^s_\lambda}] \\
                                       & = \prod_{\beta \in B} \Obj{\pi_\beta\Phi[X^s_\lambda]}              \\
        \Phi_\mathsf{Obj}[X^s_\lambda] & = \Obj{\prod_{\beta \in B} \pi_\beta\Phi[X^s_\lambda]}
    \end{align*}
    Thus $\Phi|_{X^s_\lambda}$ is essentially surjective.
    Furthermore, $\Phi|_{X^s_\lambda}$ is full and faithful as the restriction of a full and faithful functor to a full subcategory (\cref{lem:iso-slices} and \cref{cor:full-sub-cat}).\\
    Thus, $\Phi|_{X^s_\lambda}$ is a fully faithful and essentially surjective functor, thus an isomorphism in $\lfcat$
\end{proof}

\begin{proposition}
    \label{prop:iso-S}
    Let $\Phi \colon \prod_{\alpha \in A} X_\alpha \to \prod_{\beta \in B} Y_\beta$ an isomorphism of connected loop-free categories.
    Fix $\alpha \in A$ and $\beta \in B$. With
    \begin{align*}
        Y^\alpha_\beta \coloneqq \pi_\beta\Phi[X^s_\alpha] && X^\beta_\alpha \coloneqq \pi_\alpha \Phi^{-1}[Y^{\Phi(s)}_\beta] 
    \end{align*}
    The two following morphisms are inverse of each other
    \begin{align*}
        \pi_\alpha \circ \Phi^{-1} \circ  \Xi^{\Phi(s)}_\beta \colon Y^\alpha_\beta  \to X_\alpha^\beta
        & &
        \pi_\beta \circ \Phi \circ  \Xi^{s}_\alpha \colon X_\alpha^\beta \to Y^\alpha_\beta
    \end{align*}
\end{proposition}

\begin{proof}
    Let us prove that 
    \[\pi_\beta \Phi \circ  \Xi^{s}_\alpha  \pi_\alpha
        \circ \Phi^{-1} \Xi^{\Phi(s)}_\beta|_{Y^\alpha_\beta} \colon Y^\alpha_\beta \to Y^\alpha_\beta \text{
    is the identity.}\]
    First let us remark that
    $\pi_\beta \Phi
        \circ \Phi^{-1} \Xi^{\Phi(s)}_\beta = \id{Y^\alpha_\beta}$ and
    $\Xi^{s}_\alpha  \pi_\alpha|_{X^s_\alpha} = \id{X^s_\alpha}$ (\cref{cor:full-sub-cat}). But
    this contraction can only be made if we
    prove that  $\Phi^{-1} \Xi^{\Phi(s)}_\beta$ sends $Y^\alpha_\beta$ to a subcategory
    of $X^s_\alpha$.
    As
    $\pi_\beta\Phi\Xi^s_\alpha \colon X_\alpha \to Y^\alpha_\beta$ is full and essentially surjective
    by \cref{lem:iso-restriction}, it is equivalent
    to proving that $\Phi^{-1} \Xi^{\Phi(s)}_\beta\pi_\beta\Phi\Xi^s_\alpha$ sends $X_\alpha$ to a subcategory of
    $X^s_\alpha$\\

    Let $f_\alpha \in X_\alpha$ and $\lambda \in A, \lambda \neq \alpha$. By definition,
    $\pi_\lambda \id{s} = \pi_\lambda \Xi_\alpha^s f$. Thus, by \cref{cor:proj-cont},
    \begin{align*}
            &  & \pi_\lambda \Phi^{-1} (\Phi(\id{s}), \Phi(\Xi_\alpha^s f)_\beta, \beta) & = \pi_\lambda \id{s} &  & \\
        \ie &  & \pi_\lambda \Phi^{-1} \Xi^{\Phi(s)}_\beta \pi_\beta \Phi \Xi_\alpha^s f & = \pi_\lambda \id{s} &  &
    \end{align*}

    Thus, $\Phi^{-1} \Xi^{\Phi(s)}_\beta \pi_\beta \Phi \Xi_\alpha^s$ sends $ X_\alpha$ to a subcategory
    of $X^s_\alpha$. Such that,
    \begin{align*}
        \pi_\beta \Phi \circ  \Xi^{s}_\alpha  \pi_\alpha
        \circ \Phi^{-1} \Xi^{\Phi(s)}_\beta|_{Y^\alpha_\beta}
                                                              & =\pi_\beta \Phi \circ \underbrace{ \Xi^{s}_\alpha  \pi_\alpha}_{= \id{X^s_\alpha}}
        \circ \underbrace{\Phi^{-1} \Xi^{\Phi(s)}_\beta|_{Y^\alpha_\beta}}_{\text{maps to } X^s_\alpha}                                            &\\
                                                              & = \pi_\beta \circ \Phi \circ \Phi^{-1} \circ \Xi^{\Phi(s)}_\beta|_{Y^\alpha_\beta} &\\
                                                              & = \pi_\beta \Xi^{\Phi(s)}_\beta|_{Y^\alpha_\beta}                                  &\\
        \pi_\beta \Phi \circ  \Xi^{s}_\alpha  \pi_\alpha
        \circ \Phi^{-1} \Xi^{\Phi(s)}_\beta|_{Y^\alpha_\beta} & = \id{Y^\alpha_\beta}  & \text{\cref{cor:full-sub-cat}}
    \end{align*}
    Similarly, $\pi_\alpha
        \Phi^{-1} \Xi^{\Phi(s)}_\beta  \circ \pi_\beta \Phi \Xi^{s}_\alpha|_{X^\beta_\alpha}
        = \id{X^\beta_\alpha} $
\end{proof}

With this we have all we need to build the isomorphism
$ \Psi_1 = (b_\beta)_{\beta \in B} \colon \prod_{\beta \in B} Y_\beta \to \prod_{\alpha\in A, \beta \in B} X^\beta_\alpha $
and $ \Psi_2 = (a_\alpha)_{\alpha \in A} \colon \prod_{\alpha \in A} X_\alpha \to \prod_{\alpha\in A, \beta \in B} X^\beta_\alpha $
from \cref{def:refine}, \ie all the morphisms in Diagram \ref{diag:hashimoto-detailled} are isomorphisms.

\begin{proposition}
    \label{prop:iso-psi}
    With the previous notations and with
    \begin{align*}
        \Psi_1 = \biggl(\prod_{\beta \in B} \Phi^{-1} \circ \Xi^{\Phi(s)}_\beta \biggr)
        & & \Psi_2 =
        \prod_{\alpha \in A} \biggl(\prod_{\beta \in B}\biggl( \pi_\alpha\circ  \Phi^{-1} \circ \Xi^{\Phi(s)}_\beta \biggr)
        \circ \Phi \circ \Xi^s_\alpha \biggr)
    \end{align*}
    \(\Psi \circ \Phi\) and \(\Psi_2\) are isomorphisms of connected loop-free categories.
    Furthermore,
    for all $\alpha \in A$, $\beta \in B$, $\pi_\beta\Psi_1$ and $\pi_\alpha\Psi_2$ are isomorphisms.
\end{proposition}

\begin{proof}
    \begin{itemize}
        \item $\Psi_1$ is an isomorphism. Indeed, by \cref{cor:full-sub-cat} $\Xi^{\Phi(s)}_\beta \colon Y_\beta \to Y_\beta^{\Phi(s)}$ is an isomorphism.
              By \ref{lem:iso-restriction}, so is $\Phi^{-1} \colon  Y_\beta^{\Phi(s)} \to \prod_{\alpha \in A} X_\alpha^\beta$.
              Thus, as products of isomorphism, all arrows in the following diagram are isomorphisms.
              \[
                  \prod_{\beta \in B} Y_\beta
                  \xrightarrow{\prod_{\beta\in B}\Xi^{\Phi(s)}_\beta} \prod_{\beta \in B} Y_\beta^{\Phi(s)}
                  \xrightarrow{\prod_{\beta\in B}\Phi^{-1}} \prod_{\beta\in B}\prod_{\alpha \in A} X_\alpha^\beta\]

        \item Let us prove that $\Psi_2$ is an isomorphism. Indeed, by
              \cref{cor:full-sub-cat}, \resp \cref{lem:iso-restriction} \resp \ref{prop:iso-S},
              the following functors are all isomorphisms:
              \[ X_\alpha  \xrightarrow{\Xi^s_\alpha} X^s_\alpha
                  \xrightarrow{\Phi} \prod_{\beta \in B} Y_\beta ^\alpha
                  \xrightarrow{\prod_{\beta \in B} \pi_\alpha\circ \Phi^{-1} \circ \Xi^{\Phi(s)}_\beta } \prod_{\beta \in B} X^\beta_\alpha\]
              It follows that $\Psi_2 = \prod_{\lambda \in A} \biggl(\prod_{\beta \in B}\biggl( \pi_\alpha\circ  \Phi^{-1} \circ \Xi^{\Phi(s)}_\beta \biggr)
                  \circ \Phi \circ \Xi^s_\alpha \biggr)$ is an isomorphism by composition and product of isomorphisms.
              Furthermore, each $\pi_\alpha\Psi_2$ is an isomorphism
    \end{itemize}
\end{proof}


Now let us prove the commutativity of the Diagram \ref{diag:hashi-small}.

\medskip

\begin{proposition}
    \label{prop:restriction-psi-eq}
    Let $\mu \in B$, $f \in \prod_{\alpha \in A} X_\alpha^\mu = \prod_{\alpha \in A} \pi_\alpha \Phi^{-1} [Y_\mu^{\Phi(s)}]$, with
    $\Psi_1$ and
    $\Psi_2$ as defined in \cref{prop:iso-psi}.
    Then for all $\beta\in B$ and $\alpha \in A$ \[\pi_\beta\pi_\alpha (\gamma^{-1} \circ \Psi_1\circ\Phi)(f) = \pi_\beta\pi_\alpha\Psi_2(f) = \begin{cases}
            f_\alpha \quad\text{if } \beta = \mu \\
            s_\alpha \quad\text{otherwise}
        \end{cases}\]
\end{proposition}

\begin{proof}
    By definition of $\gamma$, the equality of \cref{prop:restriction-psi-eq} above is equivalent to
    \[(\Psi_1\circ\Phi(f)_\beta)_\alpha = ((\gamma\circ\Psi_2(f))_\beta)_\alpha = \begin{cases}
            f_\alpha \quad\text{if } \beta = \mu \\
            s_\alpha \quad\text{otherwise}
        \end{cases}\]

    Let $f= (f_\alpha)_{\alpha \in A} \in \prod_{\lambda \in A} X_\lambda^\mu$. By \cref{lem:iso-restriction},
    $\Phi(f) \in Y_\mu^{\Phi(s)}$ with $\Phi(f) = (\Phi(\id{s}),\Phi(f)_\mu,\mu)$ and $\Phi(f)_\mu \in Y_\mu$.
    \\
    \begin{itemize}
        \item For $\Psi_1\circ \Phi = \prod_{\beta \in B}\Phi^{-1} \circ \Xi^{\Phi(s)}_\beta \circ \pi_\beta\Phi$
            \begin{align*}
                \pi_\beta \Psi_1 \circ \Phi(f)
                    & = \Phi^{-1} \circ \Xi^{\Phi(s)}_\beta \circ \Phi (f)_\beta                                           &  &                              \\
                    & = \Phi^{-1}(\id{\Phi(s)},\Phi(f)_\beta,\beta)                                                        &  & \eqcref{def:section} \\
                    & = \Phi^{-1}(\id{\Phi(s)},(\id{\Phi(s)},\Phi(f)_\mu,\mu)_\beta,\beta)                                 &  &                              \\
                \pi_\beta \Psi_1 \circ \Phi (f) & =
                    \begin{cases}
                    \Phi^{-1}(\id{\Phi(s)}) = \id{s} &\text{if } \mu \neq \beta \\
                    \Phi^{-1}(\id{\Phi(s)},\Phi(f)_\mu ,\mu) = f &\text{if } \mu = \beta
                    \end{cases} &
              \end{align*}
        \item For $\gamma \circ \Psi_2= \prod_{\alpha \in A} \biggl(\prod_{\beta \in B} \pi_\alpha \Phi^{-1} \circ \Xi^{\Phi(s)}_\beta
                  \biggr) (\Phi \circ  \Xi_\alpha^s)$
              \begin{itemize}
                  \item $\beta \neq \mu$.\\
                        By definition,  $\id{s} = \Phi^{-1}\Phi(\id{s}) \in \Phi^{-1}[Y_\mu^{\Phi(s)}] = \prod_{\alpha\in A} X^\mu_\alpha$,
                        such that for all $\alpha \in A$, $\Xi^s_\alpha f_\alpha \in \prod_{\alpha\in A} X^\mu_\alpha$ \ie
                        $\Phi \circ \Xi_\alpha^s (f_\alpha) \in Y_\mu^{\Phi(s)}$
                        and thus for all $\beta \neq \mu$, $\pi_\beta\Phi \circ \Xi_\alpha^s (f_\alpha)
                            = \Phi(\id{s})_\beta$.\\
                        Thus, for all $\beta \neq \mu$
                        \begin{align*}
                            \pi_\alpha\pi_\beta \circ \gamma \circ \Psi_2 (f)
                                                                        & = \pi_\alpha \Phi^{-1} \Xi^{\Phi(s)}_\beta \biggl( \pi_\beta \Phi \circ  \Xi_\alpha^s (f_\alpha) \biggr) \\
                                                                        & = \pi_\alpha \Phi^{-1}  \Xi^{\Phi(s)}_\beta (\Phi(\id{s})_\beta)                                         \\
                            \pi_\alpha\pi_\beta \circ \gamma \circ \Psi_2 (f) & = \pi_\alpha{\id{s}}
                        \end{align*}
                  \item $\mu = \beta$.
                        Then, $f \in \prod_{\lambda \in A} X_\lambda^\mu$ therefore $\Phi(f) = (\Phi(\id{s}),\Phi(f)_\beta,\beta)$. It follows that
                        \begin{equation}
                            \label{eq:restr-1}
                            \forall g_\beta \in Y_\beta, (\Phi(f),g_\beta,\beta) =  (\Phi(\id{s}),g_\beta,\beta)
                        \end{equation}
                        Furthermore $\pi_\alpha \Phi^{-1}\Phi(f) = \pi_\alpha\Phi^{-1} \Phi(\id{s},f_\alpha, \alpha ) = f_\alpha$. Hence,
                        \begin{align*}
                            f_\alpha & =  \pi_\alpha \Phi^{-1}(\Phi(f), (\Phi(\id{s},f_\alpha, \alpha ))_\beta,\beta)     & \text{ By \cref{cor:proj-cont}}  \\
                                     & = \pi_\alpha \Phi^{-1}(\Phi(\id{s}), (\Phi(\id{s},f_\alpha, \alpha ))_\beta,\beta) & \text{ By \cref{eq:restr-1}}     \\
                            f_\alpha & = \pi_\alpha \pi_\beta \gamma \circ \Psi_2 (f)                                     & \text{ By definition of } \Psi_2
                        \end{align*}
              \end{itemize}
    \end{itemize}
    .
\end{proof}

Now that commutativity of Diagram~\ref{diag:hashi-small} is proven, 
we extend the commutativity along one of the \(\lambda \in A\),
thus proving the commutativity of the Diagram~\ref{diag:hashimoto-detailled}.
First as explained in the outline, we will to prove the faithfulness of the restriction,
which is less trivial in our case.

\begin{lemma}
    \label{lem:faith-restriction}
    With the previous notation, $\gamma \colon \prod_{\beta \in B} \prod_{\alpha \in A} 
    X_\alpha^\beta \to \prod_{\alpha \in A}  \prod_{\beta \in B} X_\alpha^\beta$ the natural isomorphism
    and with $\widetilde{s_\lambda}$ the singleton category for a given $\lambda \in A$ , the following restrictions are faithful functors.
    \begin{align*}
        \pi_{\lambda}\gamma^{-1}\Psi_1\circ \Phi|_{X^s_{\lambda}} &
                                                                  & \pi_{\set{\alpha \in A | \alpha\neq\lambda}}\gamma^{-1}\Psi_1\circ \Phi|_{\prod_{\alpha \neq \lambda }X_\alpha \times \widetilde{s_\lambda}} \\
        \pi_{\lambda}\Psi_2|_{X^s_{\lambda}}                      &
                                                                  & \pi_{\set{\alpha \in A | \alpha\neq\lambda}}\Psi_2|_{\prod_{\alpha \neq \lambda }X_\alpha \times \widetilde{s_\lambda}}
    \end{align*}
\end{lemma}

\begin{proof}
    \begin{itemize}
        \item $\Psi_1 \circ \Phi = \biggl(\prod_{\beta \in B} \Phi^{-1} \circ \Xi^{\Phi(s)}_\beta \biggr)\circ \Phi$ \\
              Let $f = (f_\alpha)_{\alpha \in A} \in X^s_\lambda$,
              then $\pi_\alpha f = \pi_\alpha \id{s}$ for all $\alpha \neq \lambda$.
              Thus, by \cref{cor:proj-cont}, for all $\beta \in B$ and all $\alpha \neq \lambda$,
              \[\pi_\alpha \Phi^{-1}(\Phi(\id{s}),\Phi(f)_\beta,\beta) = \pi_\alpha \id{s}\]
              \ie $\pi_\alpha \pi_\beta \Psi_1 \circ \Phi (f) = \pi_\alpha \id{s}$.
              Thus
              \begin{align*}
                  \Psi_1(f) \circ \Phi
                                      & = (\id{s},\bigl(\pi_\beta  \Psi_1(f)\bigr)_\lambda ,\lambda)_{\beta \in B} \\
                                      & = (\Xi^s_\lambda(\pi_\lambda\pi_\beta  \Psi_1(f)))_{\beta \in B}           \\
                                      & = \biggl(\prod_{\beta \in B}\Xi^s_\lambda\biggr)
                  (\pi_\lambda\pi_\beta  \Psi_1(f))_{\beta \in B}                                                  \\
                  \Psi_1\circ \Phi(f) & = \biggl(\prod_{\beta \in B}\Xi^s_\lambda\biggr) \circ
                  \biggl(\pi_\lambda \gamma^{-1} {\Psi_1\circ \Phi}\biggr) (f)
              \end{align*}

              This is true for all $f \in X^s_\lambda$, thus: $$\Psi_1\circ \Phi|_{X^s_\lambda} = \biggl(\prod_{\beta \in B}\Xi^s_\lambda\biggr)
                  \circ \biggl(\pi_\lambda \gamma^{-1} {\Psi_1\circ \Phi}\biggr)|_{X^s_\lambda}$$
              By faithfulness of $\Psi_1\circ \Phi$ (\cref{prop:iso-psi}) and $\Xi^s_\lambda$ (\cref{lem:iso-slices}),
              this implies $\pi_\lambda \gamma^{-1} {\Psi_1\circ \Phi}|_{X^s_\lambda}$ faithful.\\

              Now let $f \in (f_\alpha)_{\alpha \in A} \in \prod_{\alpha \neq \lambda }X_\alpha \times \widetilde{s_\lambda}$,
              then $\pi_\lambda f = \pi_\lambda \id{s}$. By \cref{cor:proj-cont}, this implies for all
              $\beta \in B$
              \[\pi_\lambda \Phi^{-1}(\Phi(\id{s}),\Phi(f)_\beta,\beta) = \pi_\lambda \id{s}\]
              \ie $\pi_\lambda\pi_\beta \Psi_1\circ \Phi (f) = \pi_\lambda \id{s}$
              such that
              \begin{align*}
                  \Psi_1\circ \Phi(f)
                                      & = (\pi_\beta  \Psi_1\circ \Phi(f), {\id{s}}_\lambda,\lambda)_{\beta \in B}                                      & \eqcref{cor:proj-cont} \\
                                      & = ((\pi_\alpha\pi_\beta \Psi_1\circ \Phi(f))_{\alpha\in A}, {\id{s}}_\lambda,\lambda)_{\beta \in B}                                               \\
                  \Psi_1\circ \Phi(f) & = ((\pi_\beta\pi_\alpha \gamma^{-1} \Psi_1\circ \Phi(f))_{\alpha\in A}, {\id{s}}_\lambda,\lambda)_{\beta \in B}
              \end{align*}
              Faithfulness of $\Psi_1\circ \Phi$ then implies the faithfulness of $((\pi_\beta \pi_\alpha \gamma^{-1}\Psi_1\circ \Phi(.))_{\alpha \neq \lambda})_{\beta \in B}$, \ie
              the faithfulness of $\pi_{\alpha \neq \lambda} \gamma^{-1}\Psi_1\circ \Phi$ on the subcategory $\prod_{\alpha \neq \lambda }X_\alpha \times \widetilde{s_\lambda}$

        \item  $\Psi_2 =
                  \prod_{\alpha \in A} \biggl(\prod_{\beta \in B}\biggl( \pi_\alpha\circ  \Phi^{-1} \circ \Xi^{\Phi(s)}_\beta \biggr)
                  \circ \Phi \circ \Xi^s_\alpha \biggr)$ \\
              By \cref{prop:iso-psi}, each $\pi_\alpha\Psi_2$ is an isomorphism, thus a fortiori, the restriction $\pi_\alpha\Psi_2|_{X_\alpha^s}$ is faithful. Hence,
              $\pi_{\set{\alpha \in A | \alpha\neq\lambda}}\Psi_2|_{\prod_{\alpha \neq \lambda }X_\alpha \times \widetilde{s_\lambda}}$ is faithful as a product of faithful functors.
    \end{itemize}
\end{proof}

\begin{proposition}
    \label{prop:final}
    With the previous notation, $\Psi_1 \circ \Phi = \gamma \circ \Psi_2$
\end{proposition}

\begin{proof}
    Let $\Psi \in \{\Psi_2, \gamma^{-1}\circ \Psi_1 \circ \Phi\}$ and $\mu \in B$.
    Let $f^\mu_\lambda \in X_\lambda^\mu$.
    By \cref{prop:restriction-psi-eq}, $\Psi_2$ and $\gamma^{-1}\circ \Psi_1 \circ \Phi$ are equal when we restrict to $\prod_{\alpha\in A} X^\mu_\alpha$.
    As explained above, we wish to extend this equality to the full domain.
    \cref{prop:restriction-psi-eq} also implies that for all $g^\mu \in \prod_{\alpha \in A}X^\mu_\alpha$,
    $\pi_\lambda\Psi(g^\mu,f^\mu_\lambda,\lambda) = \pi_\lambda\Psi(s,f^\mu_\lambda,\lambda)$.
    Thus, by \cref{lem:ext-iso}, for all $f_\lambda \in X_\lambda$, for all $g^\mu \in \prod_{\alpha \in A}X^\mu_\alpha$
    \begin{align}
        \label{eq:indep}
         & \pi_\lambda\Psi(g^\mu,f_\lambda,\lambda) = \pi_\lambda\Psi(s,f_\lambda,\lambda) = \pi_\lambda\Psi\Xi^s_\lambda(f_\lambda) \overset{def}{\coloneqq} \Psi_\lambda(f_\lambda)
    \end{align}
    We want to prove that for all $\lambda \in A$, $\mu\in B$, $g^\mu \in \prod_{\alpha\in A} X_\alpha^\mu$, $f_\lambda \in X_\lambda$,
    the projection on any $\alpha \neq \lambda$ of $\Psi(g^\mu,f_\alpha,\alpha)$
    depends only on $g^\mu$, \ie
    \begin{align*}
        (g^\mu,f_\lambda,\lambda) & = \Psi^{-1}(\Psi(g^\mu),\Psi_\lambda(f_\lambda),\lambda)
    \end{align*}
    \begin{itemize}
        \item For $\alpha \neq \lambda$,
              $\pi_\alpha\Psi^{-1} \Psi(g^\mu,f_\lambda,\lambda) = g^\mu_\alpha = \pi_\alpha \Psi^{-1} \Psi(g^\mu)$. Thus, by \cref{cor:proj-cont},
              \begin{equation*}
                  \pi_\alpha \Psi^{-1}(\Psi(g^\mu),\Psi(g^\mu,f_\lambda,\lambda)_\lambda,\lambda) =  g^\mu_\alpha
              \end{equation*}
              By definition of $\Psi_\lambda(f_\lambda)$ (\cref{eq:indep})
              \begin{align}
                  \label{eq:proj-1}
                  \pi_\alpha \Psi^{-1}(\Psi(g^\mu),\Psi_\lambda(f_\lambda),\lambda) & = g^\mu_\alpha
              \end{align}
        \item For $\lambda$,
              By \cref{eq:proj-1}, $\Psi^{-1}(\Psi(g^\mu),\Psi_\lambda(f_\lambda),\lambda) \in (X^\mu_\alpha, X_\lambda,\lambda)$, such that
              by \cref{eq:indep}
              \begin{align*}
                  \pi_\lambda \Psi \Xi^s_\lambda( \pi_\lambda \Psi^{-1}(\Psi(g^\mu),\Psi_\lambda(f_\lambda),\lambda)) & = \pi_\lambda
                  \Psi\Psi^{-1}(\Psi(g^\mu),\Psi_\lambda(f_\lambda),\lambda)                                          &                                              &                              \\
                                                                                                                      & = \Psi_\lambda(f_\lambda)                    &                            & \\
                  \pi_\lambda \Psi \Xi^s_\lambda( \pi_\lambda \Psi^{-1}(\Psi(g^\mu),\Psi_\lambda(f_\lambda),\lambda)) & = \pi_\lambda \Psi \Xi^s_\lambda (f_\lambda)
                                                                                                                      &                                              & \eqcref{eq:indep}
              \end{align*}
              By faithfulness of $\pi_\lambda \Psi \Xi^s_\lambda$ (\cref{lem:faith-restriction}),
              \begin{equation}
                  \label{eq:proj-2}
                  \pi_\lambda \Psi^{-1}(\Psi(g^\mu),\Psi_\lambda(f_\lambda),\lambda) = f_\lambda
              \end{equation}
    \end{itemize}
    By \cref{eq:proj-1} and \cref{eq:proj-2}, we have $
        (g^\mu,f_\lambda,\lambda) = \Psi^{-1}(\Psi(g^\mu),\Psi_\lambda(f_\lambda),\lambda)$ \ie
    \begin{align}
        \label{eq:first-step-extension}
        \Psi(g^\mu,f_\lambda,\lambda) & = (\Psi(g^\mu),\Psi_\lambda(f_\lambda),\lambda)
    \end{align}
    So far \cref{eq:first-step-extension} only holds when $g^\mu \in \prod_{\alpha\in A} X^\mu_\alpha$. Let us prove
    that it is in fact valid for all
    $g \in \prod_{\alpha \in A} X_\alpha$, \ie for all
    $\lambda\in A$ and for all $f_\lambda \in X_\lambda$
    \begin{align*}
        \Psi(g,f_\lambda,\lambda) & = (\Psi(g),\Psi_\lambda(f_\lambda),\lambda)
    \end{align*}
    Let $\rho \in A$, $\rho \neq \lambda$.
    By \cref{eq:first-step-extension}, for all $f_\lambda \in X_\lambda$,
    $\pi_\rho\Psi(g^\mu,f_\lambda,\lambda) = \pi_\rho\Psi(g^\mu)$.
    By \cref{lem:ext-iso} this implies, for all $g\in \prod_{\alpha \in A} X_\alpha$, for all $f_\lambda \in X_\lambda$
    \begin{equation}
        \label{eq:indep-2}
        \pi_{\bar{\lambda}}\Psi(g,f_\lambda,\lambda) = \pi_{\bar{\lambda}}\Psi(g) \coloneqq \Psi_{\bar{\lambda}}(g)
    \end{equation}
    where $\pi_{\bar{\lambda}}$ is the projection on $A\setminus\{\lambda\}$.\\
    Thus, proving $\Psi(g,f_\lambda,\lambda) = (\Psi(g),\Psi_\lambda(f_\lambda),\lambda)$ is equivalent to proving,
    \[(\Psi(g,f_\lambda,\lambda),\Psi_\lambda(f_\lambda),\lambda) = \Psi(g,f_\lambda,\lambda)\] \ie
    \begin{align*}
        \Psi^{-1}(\Psi(g,f_\lambda,\lambda),\Psi_\lambda(f_\lambda),\lambda) = (g,f_\lambda,\lambda)
    \end{align*}
    Let us look at the different projections
    \begin{itemize}
        \item On $X_\lambda$. \\
              $\pi_\lambda\Psi^{-1}(\Psi(g,f_\lambda,\lambda)) = \pi_\lambda\Psi^{-1} (\Psi\Xi^s_\lambda(f_\lambda)) = f_\lambda$ implies by
              \cref{cor:proj-cont}
              \begin{align*}
                  \pi_\lambda\Psi^{-1}(\Psi(g,f_\lambda,\lambda), \pi_\lambda \Psi\Xi^s_\lambda(f_\lambda), \lambda) & = f_\lambda                                 \\
                  \pi_\lambda\Psi^{-1}(\Psi(g,f_\lambda,\lambda), \Psi_\lambda(f_\lambda), \lambda)                  & = f_\lambda \quad \eqcref{eq:indep}
              \end{align*}
        \item On $\prod{\alpha\neq \lambda} X_\alpha$\\
              By \cref{eq:indep-2}, for all $h \in \prod_{\alpha \in A} X_\alpha$, $\pi_{\bar{\lambda}} \Psi(h,{\id{s}}_\lambda,\lambda) 
              = \pi_{\bar{\lambda}}\Psi(h)$.
              Applying this to the element \(h=\Psi^{-1}(\Psi(g,f_\lambda,\lambda), \Psi_\lambda(f_\lambda), \lambda) \) we get
              \begin{align*}
                  \pi_{\bar{\lambda}}\Psi(\Psi^{-1}(\Psi(g,f_\lambda,\lambda), \Psi_\lambda(f_\lambda), \lambda),{\id{s}}_\lambda,\lambda)
                   & = \pi_{\bar{\lambda}}\Psi \Psi^{-1}(\Psi(g,f_\lambda,\lambda), \Psi_\lambda(f_\lambda), \lambda) &  &                             \\
                   & = \pi_{\bar{\lambda}}(\Psi(g,f_\lambda,\lambda), \Psi_\lambda(f_\lambda), \lambda)               &  &                             \\
                   & = \pi_{\bar{\lambda}}\Psi(g,f_\lambda,\lambda)                                                   &  &                             \\
                  \pi_{\bar{\lambda}}\Psi(\Psi^{-1}(\Psi(g,f_\lambda,\lambda), \Psi_\lambda(f_\lambda), \lambda),{\id{s}}_\lambda,\lambda)
                   & = \pi_{\bar{\lambda}}\Psi(g,{\id{s}}_\lambda,\lambda)                                            &  & \eqcref{eq:indep-2}
              \end{align*}
              By faithfulness of $\pi_{\bar{\lambda}}\Psi|_{(\prod_{\alpha\in A} X_\alpha,\tilde{s_\lambda},\lambda)}$ (\cref{lem:faith-restriction}),
              \begin{align*}
                  (\Psi^{-1}(\Psi(g,f_\lambda,\lambda), \Psi_\lambda(f_\lambda), \lambda),{\id{s}}_\lambda,\lambda) & = (g,{\id{s}}_\lambda,\lambda) \\
                  \pi_{\bar{\lambda}}\Psi^{-1}(\Psi(g,f_\lambda,\lambda), \Psi_\lambda(f_\lambda))                   & = \pi_{\bar{\lambda}} g
              \end{align*}
    \end{itemize}
    Thus $\Psi^{-1}(\Psi(g,f_\lambda,\lambda),\Psi_\lambda(f_\lambda),\lambda) = (g,f_\lambda,\lambda)$
    and as such, for all $f_\lambda \in \lambda$, for all $g\in \prod_{\alpha\in A}X_\alpha$,
    \begin{align*}
        \Psi(g,f_\lambda,\lambda) & = (\Psi(g),\Psi_\lambda(f_\lambda),\lambda)
    \end{align*}
    As this is true for any $f,g \in \prod_{\alpha\in A} X_\alpha$, and any $\lambda$, we get
    \begin{align*}
        \Psi(f) = (\Psi_\alpha(f_\alpha))_{\alpha \in A}
    \end{align*}
    All that is left to do is to prove that $\Psi_\alpha(f_\alpha)$ is the same whether $\Psi = \Psi_2$ or $\Psi=\gamma^{-1}\circ\Psi_1\circ\Phi$.
    By \cref{eq:indep}
    \begin{align*}
        \Psi_\lambda(f_\lambda) = \pi_\lambda \Psi \Xi_\lambda^s(f_\lambda)
    \end{align*}
    But, by definition
    \begin{align*}
        \pi_\beta\pi_\lambda \Psi_2 (\Xi_\lambda^s f_\lambda)                                                                                             & =
        \pi_\beta\pi_\lambda (\Phi^{-1} \circ \Xi_\beta^{\Phi(s)} \circ \pi_\beta \circ \Phi \circ \Xi^s_\lambda) (\pi_\lambda (\Xi^s_\lambda f_\lambda)) &                                                                                                                         &                                     \\
                & = \pi_\beta\pi_\lambda (\Phi^{-1} \circ \Xi_\beta^{\Phi(s)} \circ \pi_\beta \circ \Phi \circ \Xi^s_\lambda) (f_\lambda) &   & \eqcref{lem:iso-slices} \\
                & = \pi_\beta\pi_\lambda (\Phi^{-1}\Xi_\beta^{\Phi(s)}) (\pi_\beta\Phi)(\Xi^s_\lambda f_\lambda)                          &   &                                 \\
\pi_\beta\pi_\lambda \Psi_2 (\Xi_\lambda^s f_\lambda)                & = \pi_\beta\pi_\lambda \gamma^{-1}\circ \Psi_1\circ \Phi(\Xi^s_\lambda f_\lambda)                                       &   & \text{by Definition of } \Psi_1
    \end{align*}
    This is true for all $\lambda$ and $\beta$, thus
    $\Psi_2 = \gamma^{-1} \circ \Psi_1 \circ \Phi$

\end{proof}

This concludes the proof of the commutativity of the diagram in \cref{thm:big-result}.

\bibliography{sn-bibliography}

\end{document}